\documentclass[11pt]{article}

\usepackage[utf8]{inputenc}
\usepackage[english]{babel}
\usepackage[margin=1in]{geometry}	
\usepackage{graphicx}
\usepackage{subcaption}
\usepackage{fancyhdr}
\usepackage{amsfonts,amssymb,amsthm,amsmath}
\usepackage{mathtools}
\usepackage{enumitem}
\usepackage[percent]{overpic}
\usepackage{tikz}
\usepackage{mathrsfs}
\usepackage{ifthen,tabularx,graphicx,multirow}
\usepackage{xargs}
\usepackage[colorinlistoftodos,prependcaption,textsize=tiny]{todonotes}
\usepackage{tcolorbox}
\usepackage{enumerate}
\usepackage{comment}
\usepackage{algorithm}
\usepackage{hyperref}
\usepackage{algpseudocode}
\usepackage{booktabs}
\usepackage{varwidth}
\usepackage[capitalise]{cleveref}
\usepackage{cite}

\DeclareMathOperator{\Var}{Var}

\definecolor{rev1}{HTML}{cb270f}
\definecolor{rev2}{HTML}{1c8235}

\makeatletter 
\newcounter{algorithmicH}
\let\oldalgorithmic\algorithmic
\renewcommand{\algorithmic}{
  \stepcounter{algorithmicH}
  \oldalgorithmic}
\renewcommand{\theHALG@line}{ALG@line.\thealgorithmicH.\arabic{ALG@line}}
\makeatother

\newcommand*{\spn}{{\mathrm{span}}}

\newcommand{\R}{\mathbb{R}}
\newcommand{\N}{\mathbb{N}}

\newcommand{\intd}{\mathrm{d}}

\newcommand{\Av}{\mathbf{A}}

\newcommand{\Fv}{\mathbf{F}}

\newcommand{\Uv}{\mathbf{U}}
\newcommand{\Vv}{\mathbf{V}}

\newcommand{\Xv}{\mathbf{X}}
\newcommand{\Yv}{\mathbf{Y}}

\newcommand{\Psiv}{\mathbf{\Psi}}
\newcommand{\Phiv}{\mathbf{\Phi}}

\newcommand{\gv}{\mathbf{g}}

\newcommand{\xv}{\mathbf{x}}
\newcommand{\yv}{\mathbf{y}}

\newcommand{\cP}{\mathcal{P}}
\newcommand{\cK}{\mathcal{K}}
\newcommand{\bG}{\mathbb{G}}

\newcommand{\bT}{\mathbb{T}}

\newcommand{\shortname}{PRONE}
\newcommand{\longname}{Petrov Regression of Nonlinear Evolution}

\newtheorem{theorem}{\bf Theorem}[section]

\newtheorem{corollary}{\bf Corollary}[section]
\newtheorem{proposition}{\bf Proposition}[section]
\newtheorem{lemma}{\bf Lemma}[section]
\newtheorem{example}{\bf Example}[section]
\newtheorem{remark}{\bf Remark}[section]

\begin{document}

\title{\shortname~: Petrov-Galerkin Operator Learning Unifies DMD, SINDy \& Koopmanism}

\author{
Matthew J. Colbrook\textsuperscript{1}\quad
April Herwig\textsuperscript{1}\thanks{Corresponding author: \href{mailto:ahlh2@cam.ac.uk}{ahlh2@cam.ac.uk}}\quad
J. Nathan Kutz\textsuperscript{2}\\[0.5em]
\small\textsuperscript{1}Department of Applied Mathematics and Theoretical Physics, University of Cambridge\\
\small Cambridge CB3 0WA, UK\\
\small\textsuperscript{2}Autodesk Research, London WC2N 4HN, UK
}

\date{}

\maketitle

\begin{abstract}
Data-driven dynamics often asks how to linearize a nonlinear system.  We ask
instead: which observables should be advanced, and where should their futures
live?  This leads to \longname{} (\shortname{}), a Petrov--Galerkin regression
framework based on
$
    \Psi(\Xv)K \approx \Phi(\Yv),
$
with distinct trial and test dictionaries.  In this form, DMD, EDMD, SINDy,
Koopman regression, sparse regression, and low-rank regression become variants
of one construction: different dictionaries, weights, and constraints.
We keep the linear algebra of Koopman learning, but drop
the artificial requirement that a finite model map a dictionary into itself.
With this asymmetry, eigenmodes are no longer the right objects. Instead, we use singular modes, which identify the observable combinations captured by the data,
their projected futures, and the strength of the coupling between the two
spaces. We identify the limiting projected operator and prove \(L^2\) convergence of
the resulting nonlinear predictor.  We give examples from chaotic maps, the double
gyre, a pitching-airfoil wake, and Lorenz--63, where \shortname{} outperforms DeepONets, Fourier neural operators, and reservoir computers with considerably fewer parameters. These examples show the same message: lift once, regress once, and
let the singular structure reveal statistics, transport, prediction, and
dimension.
\end{abstract}



\medskip

\section{Introduction}\label{sec:introduction}

We consider a discrete-time dynamical system
\[
        \xv^{k+1}=F(\xv^k), \qquad F:\Omega\to\Omega\subset\R^d.
\]
Our aim is to study its evolution from trajectory data. The obstacle is that \(F\) is typically nonlinear: a pendulum, a chemical reaction, a chaotic attractor, or a fluid wake may obey a simple law, but not one that is linear in the measured variables.

There is an old and powerful way around this difficulty.  Instead of looking
only at the state $\xv$, look at functions of the state.  These functions are
called observables.  The temperature at a point, the energy of an oscillator, a
Fourier coefficient, a polynomial feature, or a sensor reading are all
observables.  A nonlinear map on states induces a linear map on observables:
$
        \phi \mapsto \phi\circ F .
$
This is the Koopman point of view \cite{Koopman_1931}.  The price is that the
linear operator is infinite-dimensional.  The art of data-driven Koopman
computation is to choose a useful finite part of it.

Most numerical methods make this finite part square.  One chooses a dictionary
of observables, advances the data, projects back to the same dictionary, and
obtains a matrix.  Dynamic Mode Decomposition (DMD), Extended Dynamic Mode
Decomposition (EDMD), kernel variants, and many related algorithms fit naturally
into this pattern \cite{SCHMID_2010,Tu_2014,edmd,williams2015,Korda_2018}.
The resulting matrix is then often analysed through its eigenvalues and
eigenvectors.

This paper begins by changing one line of this story.  The observables whose
future we want to predict need not be the observables in which we represent that
future. That is, the right finite section of Koopman is not necessarily square.

Suppose we have snapshot pairs
\[
        \Xv =
        \begin{pmatrix}
        \xv_1^\top \\ \vdots \\ \xv_M^\top
        \end{pmatrix},
        \qquad
        \Yv =
        \begin{pmatrix}
        F(\xv_1)^\top \\ \vdots \\ F(\xv_M)^\top
        \end{pmatrix}.
\]
Choose two finite families of observables:
\[
        \Phi=\{\phi_j\}_{j=1}^{N_2},
        \qquad
        \Psi=\{\psi_i\}_{i=1}^{N_1}.
\]
The family $\Phi$ contains the quantities whose values at the next time we want
to learn.  The family $\Psi$ contains the quantities of the present state from
which we shall learn them.  The computation is simply
\begin{equation}
        \Psi(\Xv)K \approx \Phi(\Yv).
        \label{eq:intro-regression}
\end{equation}
Thus each column of $K$ is an ordinary least-squares fit.  It approximates one
future observable $\phi_j\circ F$ by a linear combination of the present
observables in $\Psi$. This is the basic object of the paper. We call the resulting framework \longname{} (\shortname{}). The construction is sketched in \Cref{fig:schematic}: lift through two
dictionaries, regress once, and read the input--output structure through the
singular modes.

Several familiar methods appear at once.  If $\Psi=\Phi=\mathrm{id}$, then
\eqref{eq:intro-regression} is DMD \cite{SCHMID_2010,Tu_2014}.  If
$\Psi=\Phi$ is a nonlinear dictionary, it is EDMD \cite{edmd,williams2015}.
If $\Phi=\mathrm{id}$ and $\Psi$ is a nonlinear library, then
\[
        K^\top \Psi(\xv)
\]
is a nonlinear approximation of $F(\xv)$ obtained by a linear regression, the
unregularized least-squares core of SINDy \cite{sindy}.  Sparse, kernel,
structure-preserving, and verified Koopman variants arise by changing the
dictionaries, the weights, the regularization, or the constraints
\cite{Brunton_2022_ModernKoopman,Colbrook_Ayton_Szoke_2023,
Colbrook_mpEDMD_2023,Colbrook_RiggedDMD_2025}.

The rectangular viewpoint also changes the modal analysis.  Eigenvectors are
the right objects for a square map from a space to itself.  But
\eqref{eq:intro-regression} maps one finite observable space into another.  In
general there need be no eigenvectors to compute.  The natural objects are
singular vectors.  Right singular modes are combinations of the observables in
$\Phi$ whose futures are most strongly captured by the chosen trial space.
Left singular modes are the corresponding projected futures in the span of
$\Psi$.  The singular values measure the strength of these pairings.

The distinction is more than cosmetic.  Eigenanalysis asks for coordinates that
reproduce themselves up to multiplication.  Singular analysis asks which
observable combinations are seen, retained, amplified, or lost by the dynamics
and by the chosen dictionary.  This is an input-output question, and it remains
meaningful when the finite Koopman section is rectangular, rank-deficient, or
deliberately asymmetric.

\begin{figure}
    \centering
    \includegraphics[width=0.9\linewidth]{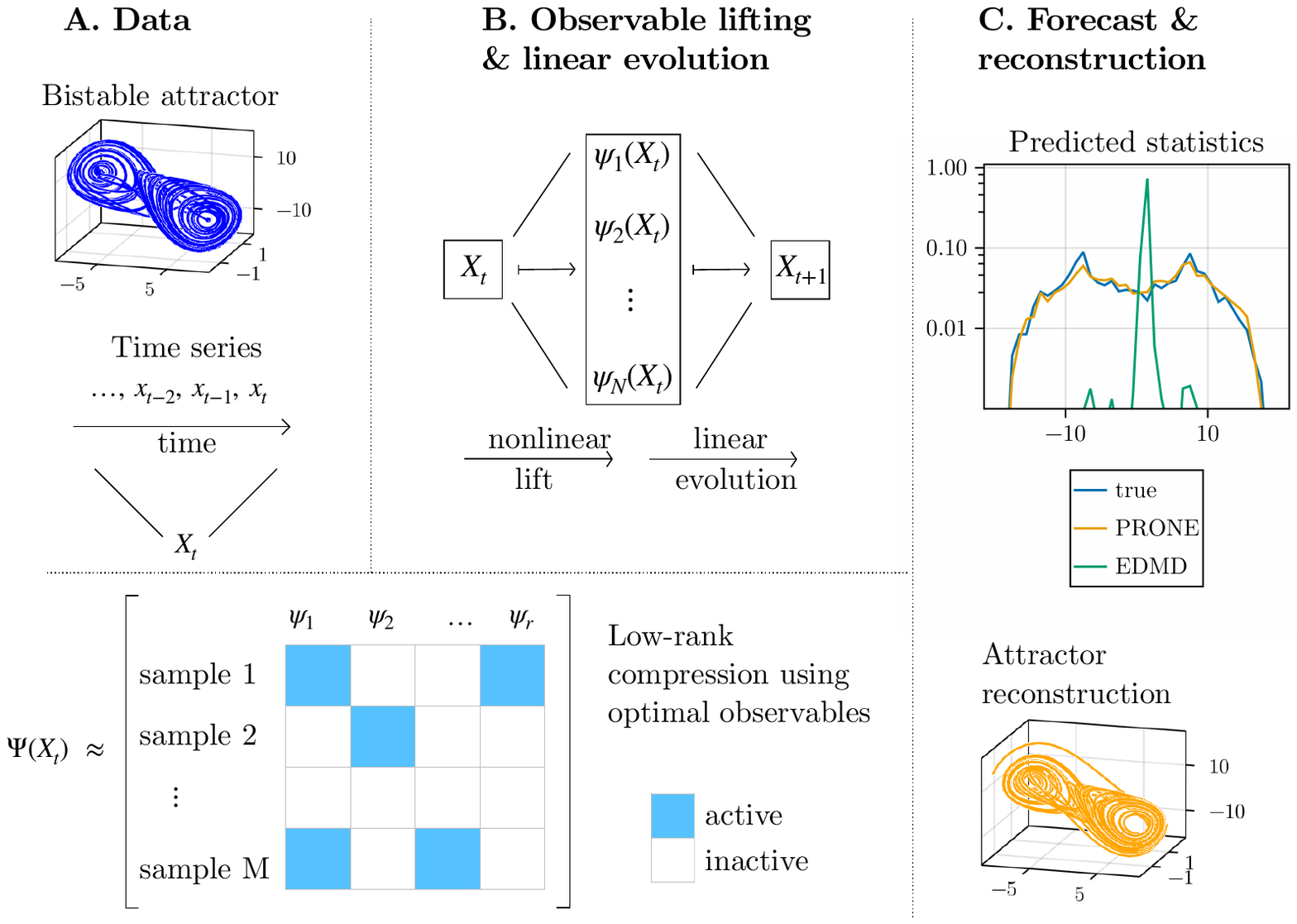}
    \caption{
    Schematic construction of \shortname{}.  Snapshot pairs are lifted through
    two possibly different observable dictionaries.  Least squares gives a
    finite Petrov--Galerkin map between the observable spaces.  When this map is
    rectangular, singular modes, rather than eigenmodes, give the natural
    input-output dynamical structures.
    }
    \label{fig:schematic}
\end{figure}

The contributions of the paper are as follows.

\begin{itemize}
\item
We formulate DMD, EDMD, SINDy, and Koopman-based regression as instances of one
Petrov--Galerkin least-squares problem between observable spaces.  The square
case is only one possibility; the rectangular case is the main freedom.

\item
We show how \shortname{} gives nonlinear predictors from linear maps in lifted
coordinates.  In the important case $\Phi=\mathrm{id}$, the learned map
$K^\top\Psi(\xv)$ is a nonlinear approximation of the dynamics $F(\xv)$.

\item
We identify the limiting operator.  Although the notation is Koopman-theoretic,
the learned predictor also exposes the adjoint, transfer-operator side of the
story.  Under quadrature and density assumptions, the finite-data regressions
converge to the corresponding projected operator, and the resulting predictors
converge in $L^2$.

\item
We develop a singular-mode analysis for rectangular Petrov--Galerkin sections.
The right modes identify observables whose evolution is captured by the chosen
space; the left modes identify their projected futures; the singular values
quantify the strength of these pairings.

\item
We demonstrate the framework on maps, flows, chaotic attractors, and fluid data.
The examples show that the same
regression captures the prediction horizon and invariant statistics of a
chaotic map, coherent transport in the double gyre, the statistical geometry of
Lorenz--63, and an effective rank for a high-dimensional airfoil wake. Moreover the time-series prediction outperforms popular operator methods, while using two orders of magnitude less parameters. 
\end{itemize}

The operator-theoretic background is classical.  Koopman observed that a
nonlinear dynamical system induces a linear operator on observables
\cite{Koopman_1931}; modern Koopman theory and DMD turned this observation into
a computational program for extracting modes, spectra, and coherent coordinates
from data \cite{Mezic_2005,Rowley_2009,SCHMID_2010,Kutz_2016_DMD}.  EDMD
expanded the dictionary of observables \cite{edmd}, kernel methods made the
feature space implicit \cite{williams2015}, and SINDy emphasized sparse
regression on nonlinear libraries \cite{sindy,Brunton_2016_SINDYc}.  Recent
developments have clarified the roles of residuals, spectral pollution,
structure preservation, and approximation limits in Koopman computations
\cite{Colbrook_Townsend_2024,Colbrook_2024_Limits,
convergentmethodskoopmanoperators,Conradie_2026_Trustworthy,
Gray_2026_DeepMDMD}.

Our message is complementary to this literature.  We do not propose another
search for Koopman eigenfunctions.  We regard the finite section itself as the
modelling object.  Once the two observable spaces are allowed to differ, the
square Koopman matrix becomes one case of a larger Petrov--Galerkin
construction.  \shortname{} is the forecasting face of this construction; the
singular-mode analysis is its interpretive face.

The rest of the paper develops this program.  The next section introduces the
Petrov--Galerkin regression, its special cases, regularization, and convergence.
\Cref{sec:singular} explains the singular modes and their interpretation.
\Cref{sec:examples} applies the method to representative nonlinear systems,
including coherent transport in the double gyre, singular coordinates on the
Lorenz attractor, and reduced modelling of a pitching-airfoil wake.

\section{Petrov--Galerkin Learning of Dynamical Systems}

We now pass from dynamics to matrices. From snapshot pairs we form two lifted data matrices, using dictionaries that may differ, and compute the best linear map between them in least squares. This map is a Petrov--Galerkin section of the evolution between observable spaces; it need not be square. We then identify the limiting operator, prove convergence of the associated predictor, and discuss the two regularizations most important in computation: sparsity and low rank.

\subsection{A Linear Framework for Nonlinear Problems}

Wishing to perform only a simple linear-regression-based learning method but still having the desire to capture nonlinear effects, we nonlinearly embed $\Xv$ and $\Yv$ into observable spaces:
\begin{equation*}
	\Psi (\Xv) = \begin{pmatrix}
		\psi_1 (\xv_1) & \ldots & \psi_{N_1} (\xv_1) \\ 
		\vdots & \ddots & \vdots \\ 
		\psi_1 (\xv_M) & \ldots & \psi_{N_1} (\xv_M)
		\end{pmatrix} \in \R^{M \times N_1} 
\end{equation*}
and 
\begin{equation*}
	\Phi (\Yv) = \begin{pmatrix}
		\phi_1 (F(\xv_1)) & \ldots & \phi_{N_2} (F(\xv_1)) \\ 
		\vdots & \ddots & \vdots \\ 
		\phi_1 (F(\xv_M)) & \ldots & \phi_{N_2} (F(\xv_M))
		\end{pmatrix} \in \R^{M \times N_2} . 
\end{equation*}
The functions $\psi_i,\phi_j \in L^2(\Omega;\R)$ are real-valued observables, and the collections $\Psi$ and $\Phi$ are dictionaries. The inner product is the real Lebesgue $L^2$ inner product $\langle f,g\rangle=\int_\Omega f g\,\mathrm{d}m$. In some cases we may use a different measure $\rho$; these cases will be clearly distincted as $L^2(\Omega ; \R ; \rho)$. 
The Petrov--Galerkin freedom is that these dictionaries need not coincide: in general $N_1 \neq N_2$. After lifting the data, the nonlinear learning problem becomes the linear regression
\begin{equation}\label{eq:petrov}
    \min_{K \in \R^{N_1 \times N_2}}
    \big\| \sqrt{W} \Psi(\Xv) K - \sqrt{W} \Phi(\Yv) \big\|_F^2 
\end{equation}
where $W \in \R^{M \times M}$ is a positive diagonal weight matrix. This problem is solved by taking 
\begin{equation}\label{eq:petrov_answer}
	K = \Psi(\Xv)^\dagger \Phi(\Yv) = \left(\Psi(\Xv)^\top \,W\, \Psi(\Xv) \right)^{-1} \Psi(\Xv)^\top \,W\, \Phi (\Yv) 
\end{equation}
where $\cdot^\dagger$ denotes the Moore-Penrose pseudoinverse. 

Here is the operator behind the regression. For an observable \(g\), the Koopman operator acts as
\begin{equation*}
    \cK : L^2(\Omega;\R) \to L^2(\Omega;\R), \quad g \mapsto \cK g = g \circ F.
\end{equation*}
On a suitable \(L^2\) space this is a linear operator; boundedness requires the usual assumptions on \(F\) and the reference measure \cite{Lasota2013-yq}. The regression \eqref{eq:petrov} is a Petrov--Galerkin approximation of this operator. For instance, suppose that the samples \(\xv_m\) are drawn independently from the reference measure \(m\), and take \(W=M^{-1}I\). Then \cref{eq:petrov} is the Monte Carlo discretization of
\begin{equation}\label{eq:error}
    \min_{K \in \R^{N_1 \times N_2}}
    \sum_{i = 1}^{N_2} \int \Big| \big(\sum_{j=1}^{N_1} K_{j i} \,\psi_j (\xv)\, \big) - [\cK \phi_i] (\xv) \Big|^2 \intd \xv . 
\end{equation}
Thus each Koopman-advanced observable $\cK \phi_i$ is approximated in the trial space $\spn\{\psi_1,\ldots,\psi_{N_1}\}$, and the regression solves all $N_2$ approximation problems simultaneously.

The choice of dictionaries makes the construction flexible.

\begin{example}
	Let $\Psi = \Phi = id$ be the identity on a bounded subset $\Omega \subset \R^d$ (so that $id \in L^2(\Omega;\R)$). Then \cref{eq:petrov_answer} reduces to $K = \Xv^\dagger \Yv$, the matrix yielded by Dynamic Mode Decomposition (DMD) \cite{SCHMID_2010}. 
\end{example}
\begin{example}
	Alternatively, consider that $\Psi$ is again some arbitrary dictionary but $\Phi$ is still the identity. One then recovers (an unregularized version of) Sparse Identification of Nonlinear Dynamics (SINDy) \cite{sindy}. We mention in \cref{sec:regularize} some forms of regularization of \cref{eq:petrov} which encompass SINDy as well. 
\end{example}
\begin{example}
	Further still, if $\Psi = \Phi$ is an arbitrary dictionary then one recovers the Extended Dynamic Mode Decomposition (EDMD) \cite{edmd}. In this context the Petrov--Galerkin method reduces to a Riesz--Galerkin one. The rectangular viewpoint is more revealing: even when the dictionaries contain the same type of observables but have different sizes, for instance $\Phi \subset \Psi$, the singular structure exposes information that is invisible in the square formulation.
\end{example}

The case \(\Phi=\mathrm{id}\) is especially important, for then the regression learns the dynamics themselves. If \(\phi_i(\xv)=x_i\), then $\cK\phi_i(\xv)=\phi_i(F(\xv))=[F(\xv)]i$, and \cref{eq:error} becomes
\begin{equation*}
	\min_{K \in \R^{N_1 \times d}}
    \int \Big\| \big(\sum_{j=1}^{N_1} K_j \,\psi_j (\xv)\, \big) - F(\xv) \Big\|_{\R^d}^2 \intd \xv . 
\end{equation*}
where $K_j$ is the vector-valued $j$th row of $K$. Thus $K^\top$ maps lifted observables back to state space, giving the optimal linear reconstruction of $F$ from the nonlinear embedding $\Psi:\Omega \to \R^{N_1}$. This is the essential architecture of \longname{} (\shortname): lift nonlinearly, regress linearly, and use rank reduction to retain the dynamically relevant directions; see \cref{alg:SINDy}. In contrast with a neural network, the nonlinear features are fixed in advance rather than learned through alternating affine maps and nonlinear activations.

\begin{algorithm}[t]
\textbf{Input:} Snapshot data $\Xv,\Yv \in \R^{M \times d}$ with $\Yv=F(\Xv)$, dictionary $\Psi$, rank $r \in \N$ \\
\vspace{-4mm}
\begin{algorithmic}[1]
\State Form the lifted data matrix $\Psi(\Xv) \in \R^{M \times N}$.
\State Compute a truncated SVD
$\Psi(\Xv)\approx U\Sigma V^\top$, with
$U \in \R^{M \times r}$, $\Sigma \in \R^{r \times r}$, $V \in \R^{N \times r}$.
\State Compute the compressed regressor
$K = V\Sigma^\dagger U^\top\Yv$.
\end{algorithmic}
\textbf{Output:} The learned map $F_{N,M}(\xv)=K^\top\Psi(\xv)$.
\caption{\longname{} (\shortname)}
\label{alg:SINDy}
\end{algorithm}

\subsection{Extracting Dynamically Relevant Modes}

First consider the EDMD case $\Psi=\Phi$, where the regression produces a square finite section of the Koopman operator $\cK$. In this setting spectral language is natural: Koopman eigenfunctions organize the asymptotic growth, decay, and oscillation of observables. If
\begin{equation*}
	g = \sum_{j = 1}^{N} \gv_j \varrho_j, \quad
	\cK \varrho_j = \lambda_j \varrho_j ,\quad
\text{then}\quad
	g(x^1) = g(F(x^0)) = \sum_{j=1}^N \lambda_j \gv_j \varrho_j(x^0) .
\end{equation*}
The eigenvalues record growth, decay, and oscillation. For nonnormal finite sections, however, eigenvalues alone are not enough; pseudospectra control finite-time growth and transient amplification \cite[page 25]{Trefethen1999-tw}. This operator-theoretic picture motivates EDMD: eigenfunctions, and approximate eigenfunctions, identify observables that remain coherent under the dynamics. Thus an observable \(g\) is coherent at scale \(\epsilon\) if
$
\|g\circ F-\lambda g\|=\mathcal{O}(\epsilon),
$
with \(\lambda\) describing the leading growth, decay, or oscillation over time intervals before the accumulated defect dominates.

The top spectral information is the most useful in practice. Eigenvalues and pseudoeigenvalues of largest modulus, together with their associated modes, give a coarse description of the dynamics through observables that decay slowly, oscillate persistently, or exhibit finite-time amplification. This interpretation must be used with care: a finite EDMD matrix $K$ need not inherit the spectrum of the Koopman operator $\cK$, and this possible mismatch motivates residual and verified variants such as ResDMD \cite{Colbrook_Ayton_Szoke_2023}. Even so, the eigenvectors $v^j$ of $K$ and their associated \emph{eigenmodes} $\sum_\ell v^j_\ell \psi_\ell$ remain useful diagnostic objects in computations; see \cref{sec:examples}.

When $\Psi \neq \Phi$, this eigenmode picture breaks. The learned map $K$ is now rectangular, or at least acts between distinct observable spaces, so eigenvectors are no longer the right objects. The corresponding structure is singular: an input observable in $\spn\Phi$ is paired with its best-resolved output in $\spn\Psi$. Given a truncated singular value decomposition
\begin{equation*}
	K = S \Lambda^{1/2} Z^\top, \quad
	S \in \R^{N_1 \times r}, \quad
	\Lambda^{1/2} \in \R^{r \times r}, \quad
	Z \in \R^{N_2 \times r}, \quad
	r \leq \min\{N_1,N_2\},
\end{equation*}
we define the \emph{singular modes}
\begin{equation*}
	\varphi_k = \sum_{\ell = 1}^{N_1} S_{\ell k} \psi_\ell \quad \text{and} \quad
	\zeta_k = \sum_{\iota = 1}^{N_2} Z_{\iota k} \phi_\iota, \quad
	1 \leq k \leq r .
\end{equation*}
They are evaluated on data by $\Psi(\Xv)S$ and $\Phi(\Xv)Z$, respectively. Their interpretation is developed in \cref{sec:singular}, and their computation is summarized in \cref{alg:one_side}.

\begin{algorithm}[t]
\textbf{Input:} Snapshot pairs $\Xv,\Yv \in \R^{M \times d}$ with $\Yv=F(\Xv)$, positive diagonal weights $W \in \R^{M \times M}$, dictionaries $\Psi,\Phi$, target rank $r \in \N$ \\
\vspace{-4mm}
\begin{algorithmic}[1]
\State Compute the data matrices $\ \Psi(\Xv) \in \R^{M \times N_1}$, $\ \Phi(\Xv) \in \R^{M \times N_2}$
\Require $N_1 \gg N_2$; the opposite regime is analogous.
\State Compute a truncated SVD $\ W^{1/2} \Psiv(\Xv)\approx U \Sigma V^\top,$ $\ U \in \mathbb{R}^{M \times r}$, $\ \Sigma \in \mathbb{R}^{r \times r}$, $\ V \in \mathbb{R}^{N_1 \times r}$
\State Compute the compression $H = U^\top W^{1/2} \Phi(\Yv) \in \R^{r \times N_2}$
\State Compute an eigendecomposition $(H H^\top) L = L \Lambda$ and set $Z = H^\top L \Lambda^{-1/2}$
\end{algorithmic} \textbf{Output:} Singular values $\ \Lambda^{1/2}$ and modes evaluated at the snapshots, $\ W^{-1/2} U L$, $\ \Phi (\Xv) Z$
\caption{Singular Mode Decomposition for Petrov-Galerkin Regression $\Pi_\Psi \cK \cP_\Phi^*$}
\label{alg:one_side}
\end{algorithm}

\subsection{Convergence}

We now connect the regression matrix $K$ to its limiting operator and prove convergence of the resulting nonlinear predictor in the appropriate topology. Define the analysis and synthesis maps
\begin{equation*}
	\begin{split}
		\cP_\Psi : L^2(\Omega;\R) \to \R^{N_1},& \quad\quad 
		\cP_\Psi g = \begin{pmatrix}
			\langle \psi_1, g \rangle \\ 
			\langle \psi_2, g \rangle \\ 
			\vdots \\ 
			\langle \psi_{N_1}, g \rangle 
		\end{pmatrix}, \\  
		\cP_\Psi^* : \R^{N_1} \to L^2(\Omega;\R),& \quad\quad 
		\cP_\Psi^* c = \sum_{\ell = 1}^{N_1} c_\ell \psi_\ell
	\end{split}
\end{equation*}
with analogous definitions for $\Phi$. If $\Psi$ is orthonormal, then $\cP_\Psi^*\cP_\Psi$ is the orthogonal projection onto $\spn\{\psi_1,\ldots,\psi_{N_1}\}$. For a general linearly independent dictionary, the same projection is
\begin{equation}\label{eq:orth_projector}
		\Pi_\Psi = \cP_\Psi^*\, (\cP_\Psi \cP_\Psi^*)^{-1}\, \cP_\Psi : L^2(\Omega;\R) \to L^2(\Omega;\R) 
\end{equation}
(and analogously for $\Phi$). 
Since \cref{eq:error} is approximated from data by quadrature, the snapshots should sample the reference measure through a convergent quadrature rule: Monte Carlo in high dimension, and higher-order rules such as Gauss--Legendre when the dimension permits. We also assume that the trial data matrix $\Psi(\Xv) \in \R^{M \times N_1}$ has full column rank for all sufficiently large $M$; with positive diagonal weights this is equivalent to full column rank of $W^{1/2}\Psi(\Xv)$. This is the finite-data counterpart of $L^2$-linear independence of the dictionary functions $\psi_j$, and in computation it is enforced by an SVD, with unresolved or nearly dependent directions discarded.
\begin{proposition}\label{prop:petrov}
		Let $\Psi$ and $\Phi$ be dictionaries whose elements are linearly independent in $L^2(\Omega;\R)$. Suppose that the snapshots $\{\xv_m\}_{m=1}^M$ and weights $W$ define a convergent quadrature rule for the relevant inner products. Let $K$ be the least-squares solution of \cref{eq:petrov}, and set
	\begin{equation}\label{eq:inner}
			G_\Psi = \Psi(\Xv)^\top W \Psi(\Xv), \quad 
			G_\Phi = \Phi(\Xv)^\top W \Phi(\Xv), \quad
			A = \Psi(\Xv)^\top W \Phi(\Yv) . 
	\end{equation}
	Then 
	\begin{align}\label{eq:one_side}
		\lim_{M \to \infty} \cP_\Psi^* K 
		&= \Pi_\Psi \cK \cP_\Phi^* ,\\
		\label{eq:both_sides}
		\lim_{M \to \infty} \cP_\Psi^* G_\Psi^{-1} A G_\Phi^{-1} \cP_\Phi 
		&= \Pi_\Psi \cK \Pi_\Phi . 
	\end{align}
\end{proposition}

\begin{corollary}\label{cor:learned_map}
		Assume the conditions of \cref{prop:petrov} are satisfied. Let $\Psi = \Psi_N$ be a sequence of dictionaries which span a dense subset of $L^2(\Omega;\R)$ in the limit $N \to \infty$.  Let $\Phi=id$ be the coordinate dictionary. The corresponding \shortname{} predictor is
  $
  		F_{N,M}(\xv)=K^\top\Psi_N(\xv).
  $
	Then 
	$
		\lim_{N \to \infty} \lim_{M \to \infty} \big\| F - F_{N,M} \big\|_{L^2(\Omega;\R^d)} = 0 . 
	$
\end{corollary}
\begin{remark}
	This result clarifies what data-based predictors such as \shortname{}, SINDy \cite{sindy}, Exact DMD \cite{Tu_2014}, and kernel EDMD \cite{williams2015} actually learn. Although the construction can be written in Koopman language, the fitted predictor is tied to the adjoint action: it represents the transfer operator, adjoint to Koopman, projected through the chosen trial and test dictionaries.
\end{remark}
	\begin{example}[Pointwise convergence in \cref{cor:learned_map} is not guaranteed]
	  	Let $\Omega=[0,1)$ be the one-dimensional torus and let
	  	$
	  		\Psi_N(x)=\left\{1,\cos(2\pi x),\sin(2\pi x),\ldots,
	  		\cos(2\pi N x),\sin(2\pi N x)\right\}
	  	$
	  	be the real Fourier dictionary. The theorem of DuBois--Reymond implies that there exists a bounded continuous function $F:
  	\Omega\to\Omega$ such that \cite[Prop. 3.3.5]{Grafakos2014-en}
  	$
  		\limsup_{N \to \infty} \Pi_{\Psi_N}F(\tfrac{1}{2})=\infty .
  	$
  	Thus $L^2$ convergence of the learned maps does not imply pointwise convergence, even for continuous targets and classical dictionaries.
\end{example}
  \begin{example}[Smoothness can restore pointwise convergence]
  	Keep $\Omega$ and $\Psi_N$ as in the previous example. If $F$ is differentiable and extends periodically to $[0,1]$, then $\Pi_{\Psi_N}F(\xv)\to F(\xv)$ as $N \to \infty$ for every $\xv \in [0,1)$ \cite[Cor. 3.3.9]{Grafakos2014-en}. Fourier extensions provide a practical way to recover pointwise convergence almost everywhere for functions that admit a differentiable periodic extension on a larger domain, including in higher dimensions \cite{BOYD2002118,webb2019}.
\end{example}

\begin{proof}[Proof of \cref{prop:petrov}]
	We first notice 
	\begin{equation*}
		\lim_{M \to \infty} G_\Psi 
		\ =\ \left( \langle \psi_i, \psi_j \rangle \right)_{1 \leq i,j \leq N_1} 
		\ =\ \cP_\Psi \cP_\Psi^* . 
	\end{equation*}
	Analogous statement holds for $\Phi$. Similarly, 
	\begin{equation*}
		\lim_{M \to \infty} A 
		\ =\ \big( \langle \psi_i, \cK \phi_j \rangle \big)_{\scriptscriptstyle{\substack{1 \leq i \leq N_1 \\ 1 \leq j \leq N_2}}}
		\ =\ \cP_\Psi \cK \cP_\Phi^* . 
	\end{equation*}
	Since the dictionary functions are linearly independent, the limiting Gram matrices are nonsingular, and hence matrix inversion is continuous along the sequence for all sufficiently large $M$. Therefore
	\begin{equation*}
		\cP_\Psi^* \lim_{M \to \infty} K 
		= \cP_\Psi^* \lim_{M \to \infty} G_\Psi^{-1} A 
		= \cP_\Psi^* (\cP_\Psi \cP_\Psi^*)^{-1} \cP_\Psi \cK \cP_\Phi^* , 
	\end{equation*}
	which, using \cref{eq:orth_projector}, is precisely $\Pi_\Psi \cK \cP_\Phi^*$.  The proof of \cref{eq:both_sides} is identical, using the additional convergence $G_\Phi^{-1}\to(\cP_\Phi\cP_\Phi^*)^{-1}$.
\end{proof}
\begin{proof}[Proof of \cref{cor:learned_map}]
	Let $\widetilde{ \Psi } = (\cP_\Psi \cP_\Psi^*)^{-1/2} \Psi$ be the orthonormal family induced by Gram-Schmidt. Then $\cP_{\widetilde{ \Psi }} = G_\Psi^{-1/2} \cP_\Psi$ so that 
	\begin{equation*}
		(\cP_{\widetilde{ \Psi }} \cP_{\widetilde{ \Psi }}^*)^{-1} 
		\cP_{\widetilde{ \Psi }} \cK \cP_\Phi^*  
		= \lim_{M \to \infty} G_\Psi^{-1/2} A 
		= \lim_{M \to \infty} G_\Psi^{1/2} K . 
	\end{equation*}
	Hence 
	\begin{equation*}
		\begin{split}
			\lim_{M \to \infty} F_{N,M} 
				&= \lim_{M \to \infty} K^\top \Psi 
				= \lim_{M \to \infty} K^\top G_\Psi^{1/2} G_\Psi^{-1/2} \Psi \\ 
			&\quad\quad
			= \cP_\Phi \cK^* \cP_{\widetilde{ \Psi }}^* 
			(\cP_{\widetilde{ \Psi }} \cP_{\widetilde{ \Psi }}^*)^{-1} 
			\widetilde{ \Psi } 
			= \cP_\Phi \cK^* \cP_{\widetilde{ \Psi }}^* \widetilde{ \Psi } . 
		\end{split}
	\end{equation*}
	Let $e_j (x) = [x]_j$ be the $j$th component function. Now 
	\begin{equation*}
		\left[ 
			\cP_\Phi \cK^* \cP_{\widetilde{ \Psi }}^* \widetilde{ \Psi } (\xv)
		\right]_j 
		=\ \Big\langle 
			e_j, \cK^* \sum_{\ell = 1}^{N_1} \widetilde{ \psi }_\ell (\xv) \widetilde{ \psi }_\ell
		 \Big\rangle 
		= \sum_{\ell = 1}^{N_1} \left\langle 
			e_j, \cK^* \widetilde{ \psi }_\ell
		 \right\rangle \widetilde{ \psi }_\ell (\xv) .
	\end{equation*}
	Since $\widetilde{ \Psi }$ is orthonormal we have 
	\begin{equation*}
		\sum_{\ell = 1}^{N_1} \left\langle 
			e_j, \cK^* \widetilde{ \psi }_\ell
		\right\rangle \widetilde{ \psi }_\ell (\xv)
		= \left( 
			\sum_{\ell = 1}^{N_1} \left\langle 
				\cK e_j, \widetilde{ \psi }_\ell
		 	\right\rangle \widetilde{ \psi }_\ell
		\right) (\xv)
		= \left( 
			\Pi_{\widetilde{ \Psi }} \cK e_j
		 \right) (\xv) 
		= \left[ \left( 
			\Pi_{\widetilde{ \Psi }} F
		 \right) (\xv) \right]_j 
	\end{equation*}
	Noting that $\Pi_{\widetilde{ \Psi }} = \Pi_\Psi$ and since $\Psi = \Psi_N$ spans a dense subset of $L^2(\Omega;\R)$ in the limit $N \to \infty$, the projected coordinate functions converge in $L^2(\Omega;\R)$ to the coordinate functions of $F$. This proves the claimed convergence of the learned map.
\end{proof}

\subsection{Regularization}\label{sec:regularize}

The least-squares solution of \cref{eq:petrov} is explicit, but applications often call for more structure. Sparsity is the simplest example. One may replace \cref{eq:petrov} by
\begin{equation*}
	\min_{K \in \R^{N_1 \times N_2}}
	\big\| W^{1/2}\big(\Psi(\Xv)K-\Phi(\Yv)\big)\big\|_F^2
	+ \lambda \|K\|_1
\end{equation*}
penalizes the number and size of active dictionary couplings, with $\lambda \geq 0$ setting the sparsity level. This objective generally loses the closed-form pseudoinverse solution, but it can be solved by standard sparse-regression algorithms \cite{bach2011optimizationsparsityinducingpenalties}. When $\Phi=id$ and $\Psi$ is a nonlinear library, this is the discrete-time, sparsity-promoting form of the SINDy architecture: a small set of nonlinear observables is selected to predict the next state \cite{sindy}.

Low rank gives a second regularization. It is useful when the dictionary is redundant, or nearly so, and several observables carry almost the same information. Let
\begin{equation*}
W^{1/2} \Psi(\Xv) \approx U\Sigma V^\top
\end{equation*}
be a truncated SVD, with $U \in \R^{M \times r}$, $\Sigma \in \R^{r \times r}$, and $V \in \R^{N_1 \times r}$. Then the least-squares solution is approximated by
\begin{equation*}
K \approx V\Sigma^\dagger U^\top W^{1/2} \Phi(\Yv).
\end{equation*}
In the square case $\Psi=\Phi$, one may work with the compressed matrix
\begin{equation*}
\widetilde{K}=V^\top KV=\Sigma^\dagger U^\top W^{1/2} \Phi(\Yv)V \in \R^{r \times r},
\end{equation*}
whose eigenvectors lift back to the original dictionary coordinates by multiplication with $V$. This is the computational core of Exact DMD: the dominant spectral information is obtained from an $r \times r$ matrix rather than the full dictionary matrix. For $\Psi\neq\Phi$, the same compression gives
\begin{equation*}
\widetilde{K}=V^\top K=\Sigma^\dagger U^\top W^{1/2} \Phi(\Yv) \in \R^{r \times N_2}.
\end{equation*}
Since $V$ is an isometry on the retained trial subspace, the singular values and right singular vectors of $K$ are obtained from $\widetilde{K}$, while the left singular vectors are lifted by $V$. Thus the rectangular Petrov--Galerkin singular modes can be computed from the smaller regression, as summarized in \cref{alg:one_side}.

\section{Singular Modes}\label{sec:singular}

A square Koopman matrix invites eigenvectors.  A rectangular Petrov--Galerkin matrix does not, and this is a feature rather than a defect.  This section studies the singular modes of $\Pi_\Psi \cK \cP_\Phi^*$ produced in \cref{alg:one_side}. Given an SVD $K = S \Lambda^{1/2} Z^\top$ the singular modes are the functions 
\begin{equation*}
	\varphi_k = \sum_{\ell = 1}^{N_1} S_{\ell k} \psi_\ell \quad \text{and} \quad
	\zeta_j = \sum_{\iota = 1}^{N_2} Z_{\iota j} \phi_\iota, \quad 
	1 \leq k,j \leq r .  
\end{equation*}
These can be evaluated at the snapshots via $\Psi(\Xv) S$ and $\Phi (\Xv) Z$, respectively. 

The right modes are the observables in $\spn\Phi$ whose Koopman images are best seen by the trial space; the left modes are those images after projection into $\spn\Psi$.  We show that the right modes are governed by $\cK^*\cK$ and hence by future mass, while the left modes are governed by $\cK\cK^*$ and hence by folding, compression, and preimages. This distinction is particularly important
when $\Phi$ consists only of coordinate functions, as in \shortname{} (\cref{alg:SINDy}), rather
than a full dictionary of nonlinear observables. \Cref{tab:singular_mode_interpretation}
summarizes the resulting interpretations.

\begin{table*}[t]
	\centering
	\scriptsize
	\renewcommand{\arraystretch}{1.0}
	\begin{tabularx}{\textwidth}{
		>{\raggedright\arraybackslash}p{0.18\textwidth}
		>{\raggedright\arraybackslash}X
		>{\raggedright\arraybackslash}X
		>{\raggedright\arraybackslash}X}
		\toprule
		Setting & Right singular modes $\zeta_j \in \spn \Phi$
		& Left singular modes $\varphi_j \in \spn \Psi$
		& Singular values \\
		\midrule
		General Petrov--Galerkin, $\Psi \neq \Phi$
		& Observables in $\spn \Phi$ whose Koopman images are strongly resolved
		in $\spn \Psi$.
		& Projected future images $\varphi_j = \Pi_\Psi \cK \zeta_j$ in the
		trial space.
		& Coupling strengths between selected observables and their projected
		future images. \\
		\addlinespace
		Coordinate space, $\Phi = id$
		& Orthogonal combinations of state coordinates that maximize
		future variance under the pushforward measure.
		& Nonlinear reconstructions of those evolved coordinate directions in
		$\spn \Psi$.
		& Predictive importance of each physical direction after nonlinear
		lifting. \\
		\addlinespace
		Full observable dictionary, $\Phi$ rich
		& Observable patterns, not coordinate axes; they maximize future
		$L^2$ mass within $\spn \Phi$.
		& Koopman-advanced versions of those patterns, projected into
		$\spn \Psi$.
		& Transfer strength between the two finite observable spaces. \\
		\addlinespace
		Square EDMD-type case, $\Psi = \Phi$
		& Finite-time amplification or preservation directions for observables,
		complementary to eigenmodes.
		& Paired output-side observable structures; mismatch reveals non-normal
		finite-time effects.
		& Finite-time gain under the projected Koopman operator. \\
		\addlinespace
		Enriched trial space, $\Phi \subset \Psi$ or $N_1 > N_2$
		& Coarser observables whose evolved images are tested in a richer
		trial space.
		& Higher-resolution evolved modes, useful when simple observables map
		to more complicated functions.
		& A diagnostic for whether $\Psi$ retains the evolved content of
		$\Phi$. \\
		\bottomrule
	\end{tabularx}
		\caption{
			Interpretation of the singular mode decomposition
			$K = S \Lambda^{1/2} Z^\top$. Right-mode interpretation depends on
			$\Phi$: coordinate functions give physical directions, while richer
			dictionaries give observable combinations. The left modes are best read as the
			associated projected future images in the trial dictionary $\Psi$.
		}
		\label{tab:singular_mode_interpretation}
	\end{table*}

\subsection{Right Singular Modes}

The right singular modes are eigenvectors/eigenfunctions of the operator 
\begin{equation*}
	\cP_\Phi \cK^* \Pi_\Psi \cK \cP_\Phi^* 
	\quad \text{resp.}\quad  
	\Pi_\Phi \cK^* \Pi_\Psi \cK \Pi_\Phi^* . 
\end{equation*}
Taking $\Psi = \Psi_{N_1}$ to be a sequence of dictionaries which span a dense subset of $L^2(\Omega;\R)$ in the limit $N_1 \to \infty$, we have 
\begin{equation}\label{eq:right_fold}
	\lim_{N_1 \to \infty} \cP_\Phi \cK^* \Pi_{\Psi_{N_1}} \cK \cP_\Phi^* 
	= \cP_\Phi \cK^* \cK \cP_\Phi^* 
	\in \R^{N_2 \times N_2}
\end{equation}
in operator norm (since $\Phi$ is finite). Thus, in the rich-trial-space limit, the right modes are governed by $\cK^*\cK$. We will make use of the following fact for Koopman operators (see, e.g., \cite{mattbook}). Let $F_\sharp m$ be the pushforward of the standard Lebesgue measure $m$ on $\Omega$ (i.e., $F_\sharp m (A) = m (F^{-1} (A))$ where $F^{-1} (A)$ is the preimage). 
\begin{lemma}\label{lem:mult}
	Suppose that $F_\sharp m$ is absolutely continuous with respect to $m$, and let $\rho = \frac{dF_\sharp m}{dm}$ be essentially bounded. Then
  	\begin{equation*}
  		(\cK^*\cK g)(x)=g(x)\rho(x).
  	\end{equation*}
\end{lemma}

	The proof of this lemma is immediate from the inner product representation
	$$
		\langle f, \cK^* \cK g \rangle
		= \langle \cK f, \cK g \rangle 
		= \int f \circ F \cdot g \circ F \,\mathrm{d} m 
		= \int f \cdot g \,\mathrm{d} F_\sharp m 
		= \int f \cdot g \cdot \rho \,\mathrm{d} m 
		= \langle f, g \rho \rangle . 
		$$
Combining \cref{lem:mult} with \cref{eq:right_fold}, the limiting right-mode problem is governed by
\begin{equation*}
	\cP_\Phi \cK^* \cK \cP_\Phi^* = \left( \langle 
		\phi_i, \phi_j
	\rangle_{L^2(\rho)} \right)_{1 \leq i,j \leq N_2} . 
\end{equation*}
The eigenvectors $Z_{:,k}$ therefore satisfy the Rayleigh principle
\begin{equation}\label{eq:rayleigh}
	\lambda_{k + 1} = \max \Big\{ 
		\| \phi \|_{L^2 (\rho)}^2
		\ \Big|\ \phi = \sum_{\ell = 1}^{N_2} [z]_\ell\, \phi_\ell,\;\ z^\top z = 1,\;\ 
		z \perp Z_{:,j},\;\ \forall\, 1 \leq j \leq k
	\Big\} . 
\end{equation}
In particular, when $\Phi$ is orthonormal, then 
\begin{equation*}
	\lambda_{k + 1} = \max \left\{ 
		\frac{\| \zeta \|_{L^2 (\rho)}^2}{\| \zeta \|_{L^2 (m)}^2 }
		\ \Big|\ \zeta \in \spn\, \Phi,\;\ 
		\zeta \perp \zeta_j\;\ \forall\, 1 \leq j \leq k
	\right\} . 
\end{equation*}
From this we conclude that the right singular modes $\zeta_k = \sum_{\ell = 1}^{N_2} Z_{\ell,k} \phi_\ell$ are preciseley the mutually orthogonal observables in $\spn\Phi$ with maximal future $L^2$ mass.

Arnold's cat map gives a clean illustration. Let $F:\bT^2\to\bT^2$ be
\begin{equation*}
	F 
	\begin{pmatrix}
		x \\ y
	\end{pmatrix} =
	\begin{pmatrix}
		2x + y \mod 1 \\ 
		x + y \mod 1
	\end{pmatrix} , 
\end{equation*}
a classically studied map on the torus.  Since $F$ is measure-preserving, $\rho=1$, and the Rayleigh quotient in \cref{eq:rayleigh} is identically one in the full-space limit. Finite sections can nevertheless lose modes. For the real Fourier block
\[
	C_{n,m}(x,y)=\cos(2\pi(nx+my)),\quad
	S_{n,m}(x,y)=\sin(2\pi(nx+my)),
\]
the Koopman action is
\begin{equation*}
\cK C_{n,m}=C_{2n+m,n+m},\quad
\cK S_{n,m}=S_{2n+m,n+m}.
\end{equation*}
Thus, if $\Phi$ contains all sine--cosine blocks with $|n|,|m|\leq N$ and $\Psi$ contains all blocks with $|n|,|m|\leq 3N$, every Koopman image of a test block is retained and all singular values are one. A square truncation need not have this property: blocks whose images leave the truncated box are projected away, producing zero singular values. This is the practical advantage of the Petrov--Galerkin view: a rectangular trial space can preserve the dynamically generated content that a square EDMD section discards.

\Cref{eq:rayleigh} is especially revealing when $\Phi=id$, the prediction setting of \shortname{} (\cref{alg:SINDy}). In this case the columns of $Z$ rotate the coordinate observables into directions that best capture future states: they maximize variance with respect to the pushforward density $\rho$. More precisely, if the coordinate functions are centered, then for $z \in \R^d$,
\begin{equation}\label{eq:variance}
    \Var_\rho \Big( \sum_{\ell = 1}^d [z]_\ell \,[\cdot]_\ell \Big) 
    = \| z^\top \cdot \,\|_{L^2 (\rho)}^2 
    = \| z^\top F(\cdot) \|_{L^2 (m)}^2 . 
\end{equation}
As a concrete example, take $F$ to be the inverse Arnold cat map and parameterize the torus by $[-1/2,1/2]^2_{\mathrm{per}}$. Let $\Phi=\sqrt{12}\,id$, so that the coordinate functions are normalized, and choose an orthonormal family $\Psi \subset L^2([-1/2,1/2]^2_{\mathrm{per}};\R)$ for which $\Pi_\Psi F^{-1}=F^{-1}$. The singular modes selected by \cref{alg:one_side} are shown in \cref{fig:arnold_sindy}; their analytic computation is given in \cref{sec:arnold_sindy}.

\begin{figure}
	\centering
	\begin{subfigure}{0.37\textwidth}
		\centering
		\includegraphics[width=\linewidth]{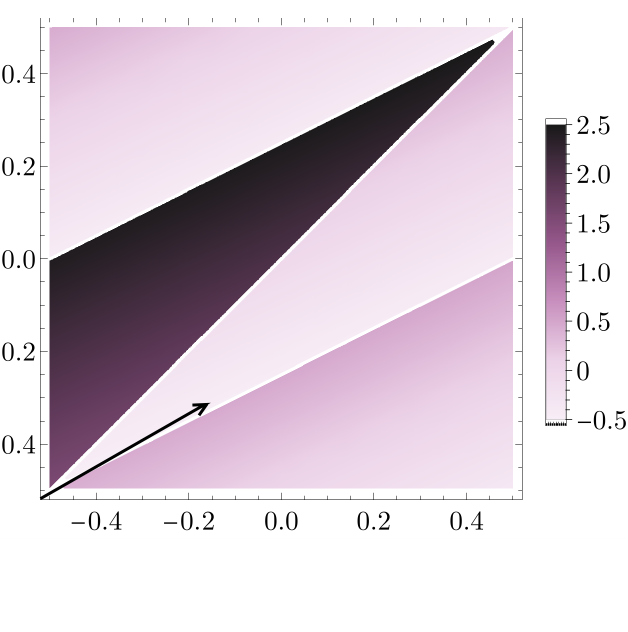}
	\end{subfigure}
	\hfill
	\begin{subfigure}{0.37\textwidth}
		\centering
		\includegraphics[width=\linewidth]{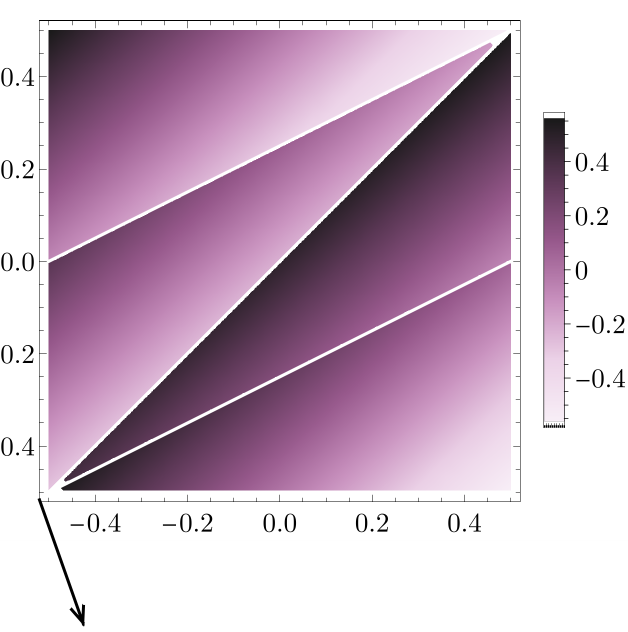}
	\end{subfigure}
	\hfill
	\begin{subfigure}{0.21\textwidth}
		\centering
		\includegraphics[width=\textwidth]{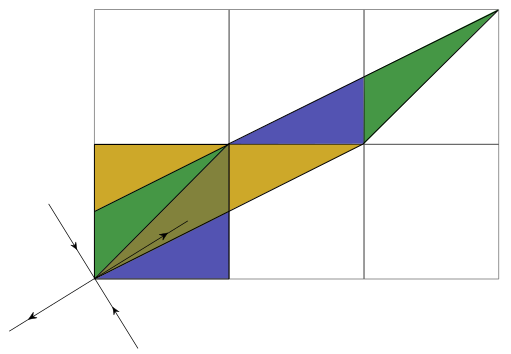}
		\vspace{8ex}
	\end{subfigure}
	\caption{
  		Singular modes for the inverse Arnold cat map. The colored fields show the right singular modes selected by \cref{alg:one_side}; the arrows indicate the corresponding coordinate directions. These directions align closely with the expanding and contracting directions of the linear map, shown schematically on the right, but they are not identical. The singular modes optimize future variance in the measure-theoretic sense of \cref{eq:variance}, rather than topological stretching alone. 
	}\label{fig:arnold_sindy}
\end{figure}

\subsection{Left Singular Modes}

The left singular modes admit a dual interpretation. Since $\varphi_j=\Pi_\Psi \cK \zeta_j$, they are the Koopman-evolved right singular modes, resolved in the trial space $\spn\Psi$. To isolate the limiting structure, consider the fully projected operator $\Pi_\Psi\cK\Pi_\Phi$. If $\Psi$ is fixed and $\Phi=\Phi_{N_2}$ becomes dense in $L^2(\Omega;\R)$ as $N_2\to\infty$, then the left singular modes are governed by $\Pi_\Psi \cK\cK^*\Pi_\Psi$. This operator has an explicit form for nonsingular, finite-to-one $C^1$ maps with nonzero Jacobian determinant almost everywhere. In that case
\begin{equation*}
	(\cK^* h)(\yv) =
	\sum_{\xv \in F^{-1}(\yv)}
	\frac{h(\xv)}{|\det D F(\xv)|} ,
\end{equation*}
and therefore
\begin{equation*}
	(\cK \cK^* h)(\xv)
	=
	\sum_{\widetilde{\xv}\, :\, F(\widetilde{\xv}) = F(\xv)}
	\frac{h(\widetilde{\xv})}{|\det D F(\widetilde{\xv})|}.
\end{equation*}
Thus the left singular modes are the observables in $\spn \Psi$ whose weighted sums over preimage fibers have maximal $L^2$ energy. Values on points that are mapped to the same future state add constructively or cancel, while the Jacobian weights emphasize regions whose volume is compressed by the map.  If $F$ is injective, this reduces to
\begin{equation*}
	(\cK \cK^* h)(\xv)
	=
	\frac{h(\xv)}{|\det D F(\xv)|},
\end{equation*}
so the left modes concentrate on directions resolved by $\Psi$ where the inverse-volume weight is large. The conceptual duality of this multiplication operator with $\cK^* \cK$ (cf. \cref{lem:mult}) should be emphasized: $\cK^*\cK$ measures future mass, while $\cK\cK^*$ measures how evolved observables fold back geometrically through the preimage. 

Returning to Arnold's cat map, the Koopman operator is orthogonal on the real $L^2$ space. The right singular modes for the map are therefore the left singular modes for the inverse map, which is precisely the relationship illustrated in \cref{fig:arnold_sindy}.

\section{Examples}\label{sec:examples}

We test the method on four problems of increasing difficulty: a chaotic cubic map, the periodically driven double gyre, the Lorenz--63 system, and a pitching-airfoil wake. Each example stresses a different aspect of the framework. For the cubic map, pointwise prediction gives way to invariant statistics. For the double gyre, singular modes reveal coherent transport. For Lorenz--63, nonlinear lifting improves finite-time prediction while preserving the statistical shape of the attractor. For the airfoil data, the singular spectrum gives a practical rank.

\subsection{Cubic map}

We begin with a one-dimensional chaotic map,
\begin{equation*}
F(x) = \alpha x(1 - x^2),
\end{equation*}
where $\alpha \approx 2.598076211353312$ makes $F$ surjective from $[0,1]$ to $[0,1]$. We train \shortname{} from a single trajectory of length $M + 1$, giving $M = 1000$ snapshot pairs, using the dictionary
\begin{equation*}
	\Psi(x)
	=
	\begin{pmatrix}
		x & x^2 & x^3 & x^4 & x^5
		& \sin(\pi x) & \sin(2\pi x)
		& \cos(\pi x) & \cos(2\pi x)
	\end{pmatrix}^{\top}.
\end{equation*}
For this example, the model class contains the dynamics exactly:
\begin{equation*}
	K_{\mathrm{true}} =
	\begin{pmatrix}
		\alpha & 0 & -\alpha & 0 & 0 & 0 & 0 & 0 & 0
	\end{pmatrix}.
\end{equation*}
Here $\Phi = id$ and there is only one coordinate direction, so the coordinate-space row of
\cref{tab:singular_mode_interpretation} reduces to a scalar test: whether the nonlinear trial
dictionary carries the future state strongly enough to preserve both trajectory-level information
and invariant statistics.

\begin{figure}
	\centering
	\includegraphics[width=0.8\linewidth]{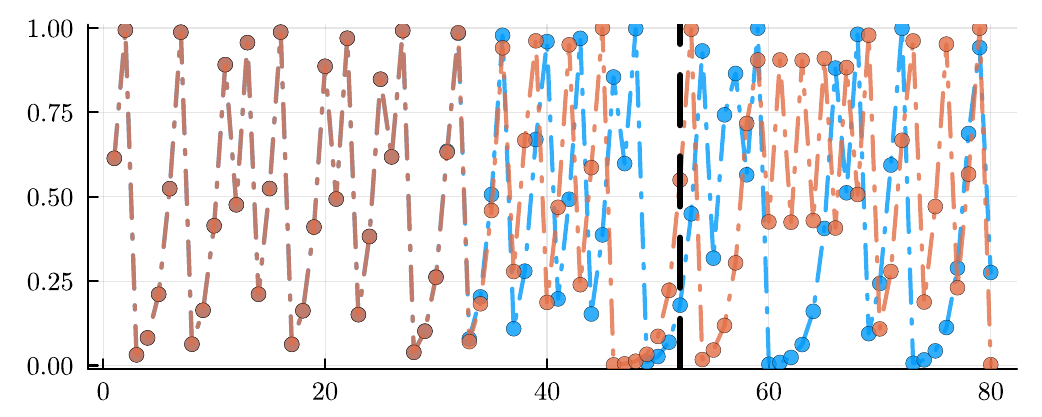}
	\caption{
		Forecast of the chaotic cubic map from a model trained with $M = 5{,}000$ snapshots. The blue curve is the true trajectory and the orange curve is the \shortname{} prediction. The trajectories agree until the Lyapunov exponent divided by machine precision, indicated by the black dashed line; after this horizon, trajectory-level agreement is no longer meaningful for a chaotic system.
	}
	\label{cubic2}
\end{figure}

\begin{figure}
	\centering
	\includegraphics[width=0.6\linewidth]{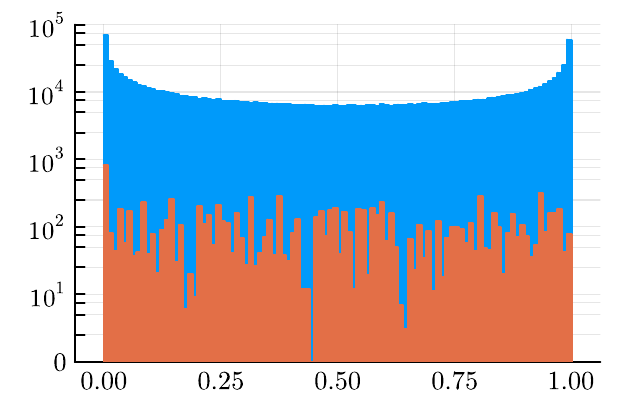}
	\caption{
		Long-time statistics for the cubic map over a trajectory of length $10^6$. The blue histogram shows the true invariant statistics on a logarithmic scale, while the orange bars show the binwise discrepancy for the \shortname{} model. Although individual trajectories separate after the Lyapunov horizon, the learned map preserves the invariant measure to within approximately $1\%$.
	}
	\label{cubic3}
\end{figure}

With $M = 5{,}000$, \shortname{} reproduces the true trajectory until the chaotic shadowing horizon set by the Lyapunov exponent and machine precision; see \cref{cubic2}. This is the strongest trajectory-level agreement one should expect in a chaotic system. The more important test is statistical. We iterate the learned model for $10^6$ steps and compare the resulting histogram with that of the true map in \cref{cubic3}. The invariant statistics agree at the percent level, showing that the Petrov--Galerkin predictor captures both short-time evolution and long-time measure.

\subsection{Periodically Driven Double Gyre}

We consider an (initially) nonautonomous flow on the domain $[0,2] \times [0,1]$ governed by 
$$
		\dot{x} = - \pi A \sin (\pi f(t, x)) \cos (\pi y),\qquad 
		\dot{y} = \pi A \cos (\pi f(t, x)) \sin (\pi y) \frac{\partial f}{\partial x} (t, x)
$$
where $f(t,x) = \alpha \sin(\omega t)x^2 + (1 - 2\alpha \sin(\omega t))x$. With $A = 0.25$, $\alpha = 0.25$, and $\omega = 2\pi$, the vector field has period one, so its time-$1$ stroboscopic map defines an autonomous discrete-time system. The flow consists of two counter-rotating gyres separated by an oscillating vertical boundary, allowing particles to exchange intermittently between the two regions \cite{Froyland_2017}. This transport mechanism is a standard failure mode for purely linear prediction methods such as DMD. We apply \shortname{} using $50$ Cartesian products of Legendre polynomials, $(x,y) \mapsto P_i(x)P_j(y)$, with $(i,j) = (0,0),(1,0),(0,1),(1,1),(2,1),\ldots$. Following the coordinate-centering lesson from the cubic-map example, we transplant the domain to $[-1,1]\times[-1,1]$ and rescale the dictionary accordingly. The regression uses $250$ Gauss--Legendre nodes in each dimension, for a total of $62{,}500$ snapshots. 

Because the time-$1$ double-gyre map is volume-preserving, the Rayleigh quotient in \cref{eq:rayleigh} is constant: in the infinite-resolution limit, no coordinate direction is intrinsically preferred. The finite computation tells a more useful story. Quadrature and dictionary truncation introduce a small nonuniformity in the numerical pushforward $\rho = \tfrac{d f_\sharp m}{d m}$, concentrating slightly more mass near the gyre cores. The right singular vectors respond exactly as \cref{eq:variance} predicts: they choose the coordinate directions with maximal variance under this approximate future distribution; see \cref{fig:gyre_sindy} and the coordinate-space interpretation in \cref{tab:singular_mode_interpretation}. 

\begin{figure}
	\centering
	\includegraphics[width=0.6\linewidth]{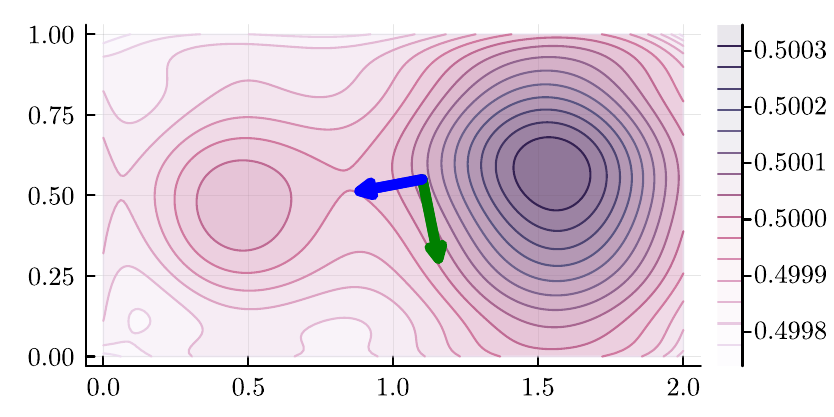}
	\caption{
		Numerical pushforward weighting for the periodically driven double gyre, together with the leading coordinate directions selected by \cref{alg:SINDy,alg:one_side}. The first direction, shown in green, has singular value $0.953$ and corresponds to the linear observable $x \mapsto -0.981x_1 - 0.193x_2$. The second direction, shown in blue, has singular value $0.929$. These directions maximize variance with respect to the computed pushforward measure, illustrating the coordinate-selection mechanism in \cref{eq:variance}.
	}
	\label{fig:gyre_sindy}
\end{figure}

\begin{figure}
	\centering
	\includegraphics[width=0.9\linewidth]{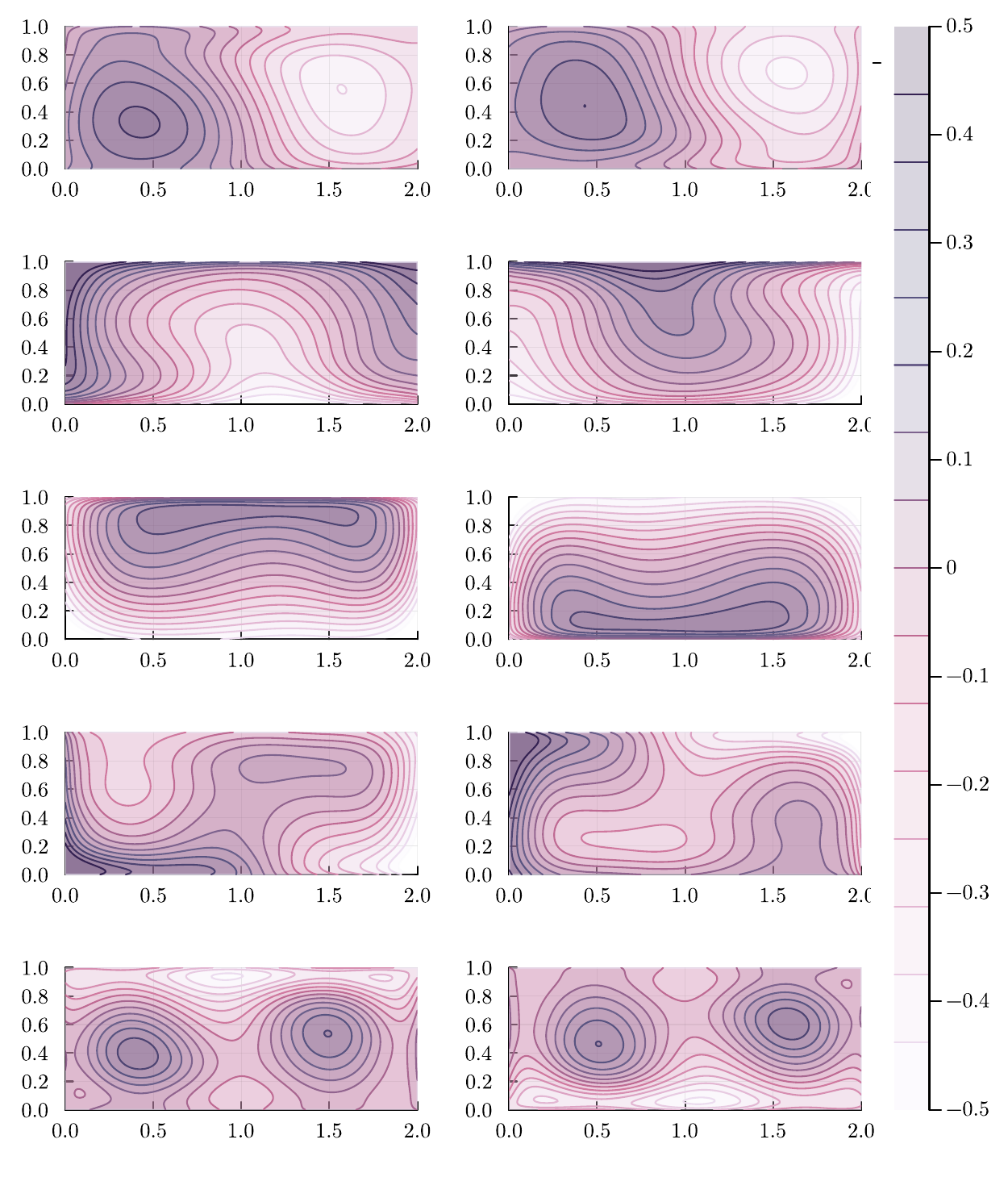}
	\caption{
		Leading singular modes for the periodically driven double gyre. The test dictionary $\Phi$ contains $50$ Cartesian products of Legendre polynomials, $(x,y) \mapsto P_i(x)P_j(y)$, transplanted to the double-gyre domain, while the trial dictionary $\Psi$ uses twice as many Legendre products to resolve the evolved observables. The top modes separate the two gyres and their transport boundary, matching the coherent-set structure expected for this flow. Higher modes refine this geometry near the upper and lower transport channels.
	}
	\label{fig:gyre_sings}
\end{figure}

We also apply \cref{alg:one_side} with the same setup but reduce $\Phi$ to $25$ Legendre-product observables. The leading singular mode, split by sign, recovers the two coherent sets reported in \cite{Froyland_2017}; see \cref{fig:gyre_sings}. Here the enriched-trial-space row of \cref{tab:singular_mode_interpretation} is used: coarser Legendre observables in $\Phi$ evolve into higher-resolution projected modes in $\Psi$. Thus the singular-mode calculation does more than predict trajectories: it extracts the transport geometry encoded by the rectangular Petrov--Galerkin regression.

\subsection{The Lorenz--63 Equations}

\begin{figure}
	\centering
	\includegraphics[width=0.99\linewidth]{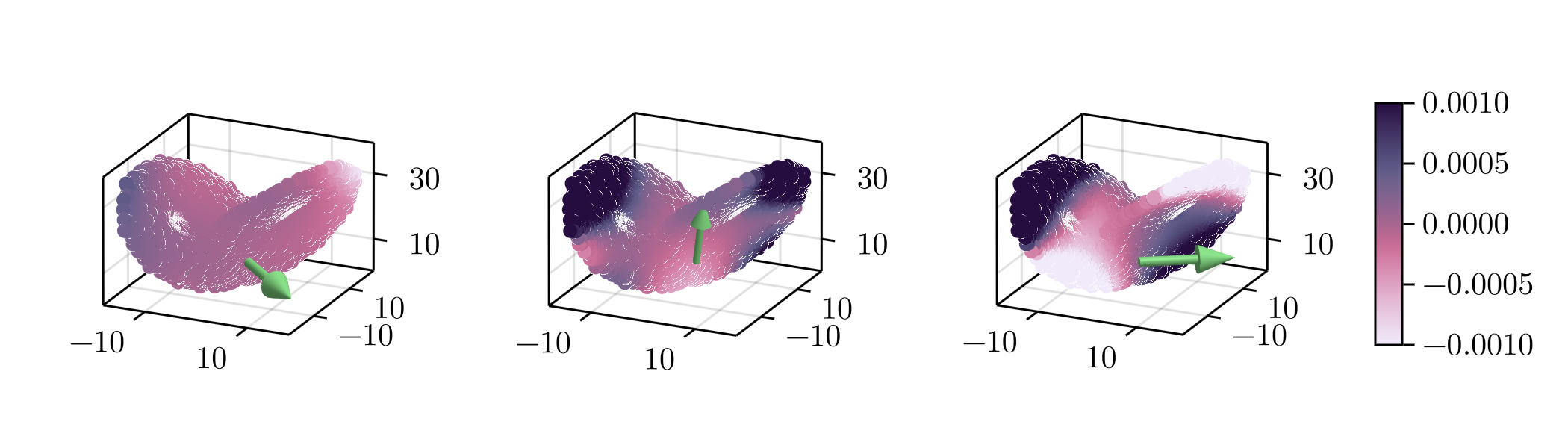}
	\caption{
		\shortname{} singular modes for Lorenz--63. Each panel shows one left singular mode evaluated on trajectory samples from the Lorenz attractor, with color indicating the mode amplitude. The green arrow gives the associated right singular direction in coordinate space. These paired objects show how the learned Petrov--Galerkin map selects physical coordinate combinations and resolves their nonlinear Koopman-advanced images on the attractor.
	}
	\label{fig:lorenz_one_side}
\end{figure}

\begin{figure}
	\centering
	\includegraphics[width=0.8\linewidth]{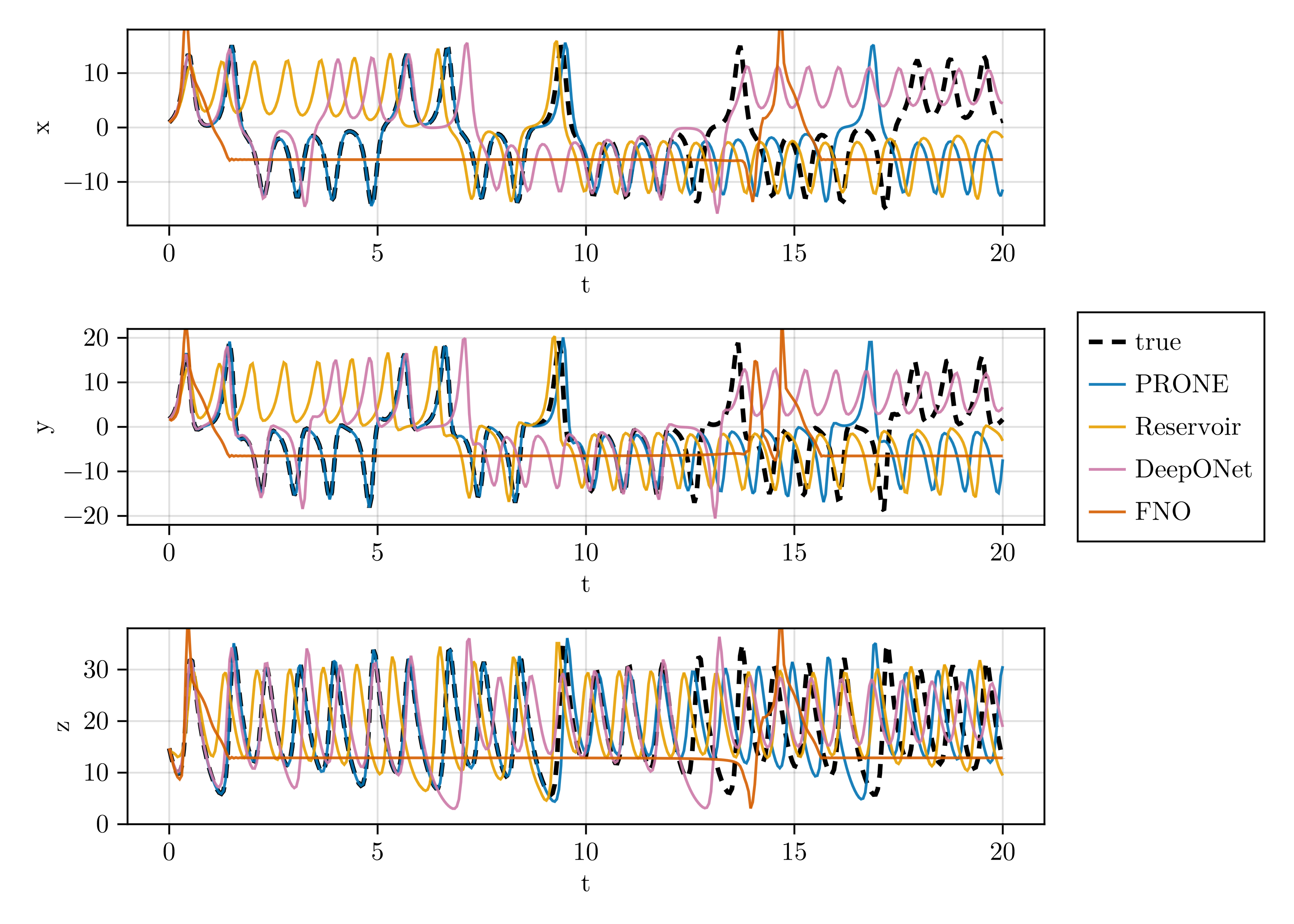}
	\caption{
		Coordinate-wise Lorenz--63 forecasts from a generic initial condition. The black dashed curve is the true trajectory, blue is \shortname{}, yellow is the reservoir computer, green is EDMD, pink is a DeepONet, orange is a Fourier neural operator. 
	}
	\label{fig:lorenz_comparison}
\end{figure}

Lorenz's 1963 equations 
\begin{equation*}
    \dot{x} = \sigma (y - x), \quad 
    \dot{y} = x (r - z) - y, \quad
    \dot{z} = x y - b z
\end{equation*}
with the classical parameters $\sigma = 10$, $r = 28$, and $b = 8/3$, provide a compact three-dimensional test case for nonlinear Petrov--Galerkin learning. We compute \cref{eq:one_side} using the polynomial dictionary $(x,y,z) \mapsto x^{p_1}y^{p_2}z^{p_3}$ with $p_1,p_2,p_3 \leq 3$, giving $64$ observables. The data consist of $10{,}000$ points drawn from $25$ trajectories initialized randomly near the origin and sampled with time step $\Delta t = 0.1$; these trajectory samples serve as ergodic quadrature points. \Cref{fig:lorenz_one_side} shows the output of \cref{alg:one_side}, with the coordinate frame selected by \shortname{} superimposed. The selected directions again maximize future variance, now on the Lorenz attractor rather than on a prescribed quadrature grid. Referring back to \cref{tab:singular_mode_interpretation}, the green arrows are right modes in coordinate space, while the colored fields are the corresponding projected future images in $\spn\Psi$.

Predictive performance for a representative initial condition is shown in \cref{fig:lorenz_comparison}, where the predictions are compared across all three Lorenz state variables. We benchmark \shortname{} against
\begin{itemize}[leftmargin=1em]
    \item a reservoir computer with a reservoir dimension $128$,
    \item a DeepONet with branch-network layer widths $(3,32,32,16)$ and trunk-network layer widths $(1,32,32,16)$,
    \item a Fourier neural operator with channel widths $(1,12,24,24,12,1)$ across successive layers.
\end{itemize}
All methods are trained on the same $M$ training samples. The DeepONet and Fourier neural operator architectures were selected through a hyperparameter sweep, subject to a budget of at most $10{,}000$ trainable parameters per model.
   
\shortname{} uses only $64$ polynomial observables and reevaluates the nonlinear lift $\Psi : \R^3 \to \R^{64}$ at every time step. This relifting is crucial: EDMD rapidly collapses toward a nearly steady trajectory. The resulting learned map tracks the true trajectory substantially longer than all other tested methods despite having (by far) the least number of parameters. We compute the normalized 50-step-ahead root mean square error  
\begin{equation*}
    \textrm{NRMSE}(x,\hat{x}) = 
    \sum_{k=1}^{50} 
        | x(t + k \Delta t) - \hat{x}(t + k \Delta t) |^2
    \,\Big/\, 
    \sum_{k=1}^{50} | x(t + k \Delta t) |^2
\end{equation*}
and average this over $150$ trajectories. The result is shown in \cref{tab:errors}. 

Finally, to demonstrate the long-term statistical behavior of the system we compute the Sinai-Ruelle-Bowen (SRB) measure which we approximate by ergodic sampling, using $150$ trajectories over $t \in [0,25]$ and binning the resulting points on a $40 \times 40 \times 40$ grid. The measures in \cref{fig:lorenz_invariant_measure} show that \shortname{} recovers the two-lobed Lorenz attractor and its dominant mass distribution, while the reservoir model captures the broad geometry but distorts the density along the wings. We compute the $L^1$ error between the true and predicted invariant densities, $\int_{\mathbb{R}^3} | \rho - \hat{\rho} |$, in \cref{tab:errors}. 

\begin{table}
    \centering
    \begin{tabular}{rcc}
         \hline
         Model & Short-term prediction error & Long-term statistics error  \\
         \hline
         \shortname{} & $8.78$ & $0.540$ \\ 
         DeepONet & $693$ & $1.43$ \\ 
         Fourier neural operator & $177$ & $2.00$ \\ 
         Reservoir computer & $136$ & $0.90$ \\ 
         \hline
    \end{tabular}
    \caption{
        Short-term and long-term errors to three significant figures for the tested models. 
    }
    \label{tab:errors}
\end{table}

\begin{figure}
	\centering
	\includegraphics[width=0.9\linewidth]{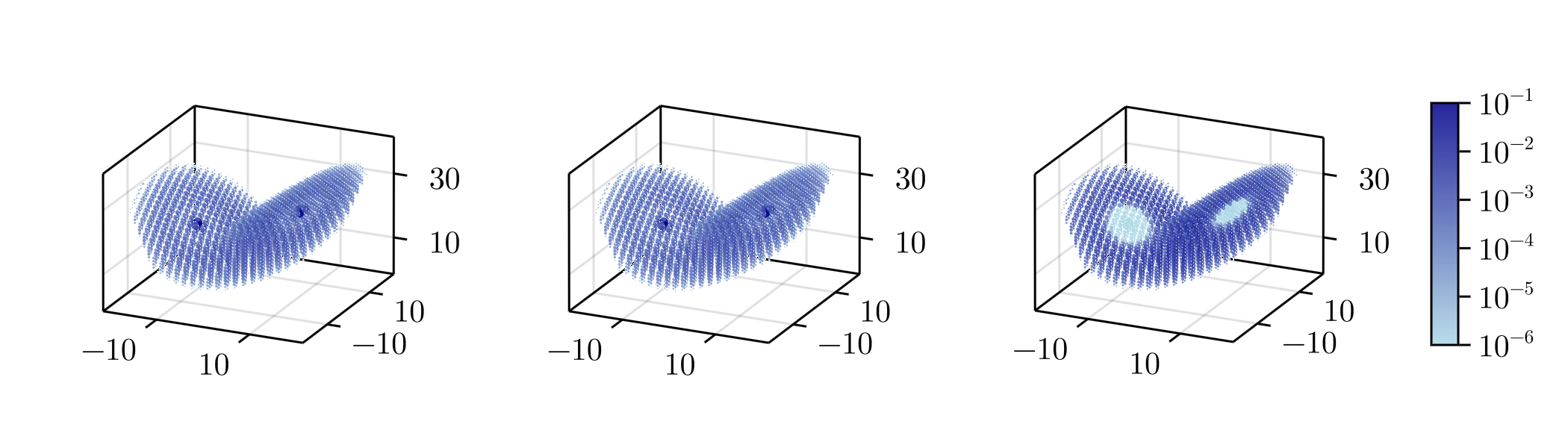}
	\caption{
		Approximate SRB measures for Lorenz--63, computed by ergodic sampling and binning on a $40 \times 40 \times 40$ grid. Left: the true flow. Middle: the measure generated by the \shortname{} predictor. Right: the measure generated by the reservoir computer. \shortname{} preserves the two-wing attractor geometry and the dominant density distribution, while the reservoir model captures the broad support but distorts the mass along the wings.
	}
	\label{fig:lorenz_invariant_measure}
\end{figure}

\subsection{Pitching Airfoil in Low-Reynolds Number Fluid}

We close with a high-dimensional fluid example: the wake of a pitching flat-plate airfoil at low Reynolds number \cite{Towne2022-ae}. Pitching airfoils are relevant to small engineered and biological flyers, where rapid wing motions generate strongly unsteady aerodynamic loads. Accurate reduced-order models are therefore essential. Each snapshot contains the $x$- and $y$-velocity components on four nested $600 \times 300$ grids. We use $M = 950$ snapshots and radial basis functions $\psi_c(x) = \exp(-\|x - c\|^2/\varepsilon)$, with $N = 90$ centers $c$ chosen by uniform subsampling of the trajectory and bandwidth $\varepsilon = 10^{-3}$. The resulting left singular modes are shown in \cref{fig:airfoil}: the leading modes resolve coherent wake structures shed from the airfoil, while higher modes capture finer downstream oscillations (see "left singular modes" in coordinate space in \cref{tab:singular_mode_interpretation}). \Cref{fig:airfoil_error} compares the singular values with the NRMSE for five-step-ahead prediction on $50$ held-out snapshots. The error drops rapidly as the first modes are added and then levels off after roughly $r = 10$, matching the decay of the singular values. Thus the singular spectrum gives a practical rank-selection rule: modes beyond this point carry little additional predictive energy.

\begin{figure}
	\centering
	\begin{subfigure}{0.8\textwidth}
		\centering
		\includegraphics[width=\linewidth]{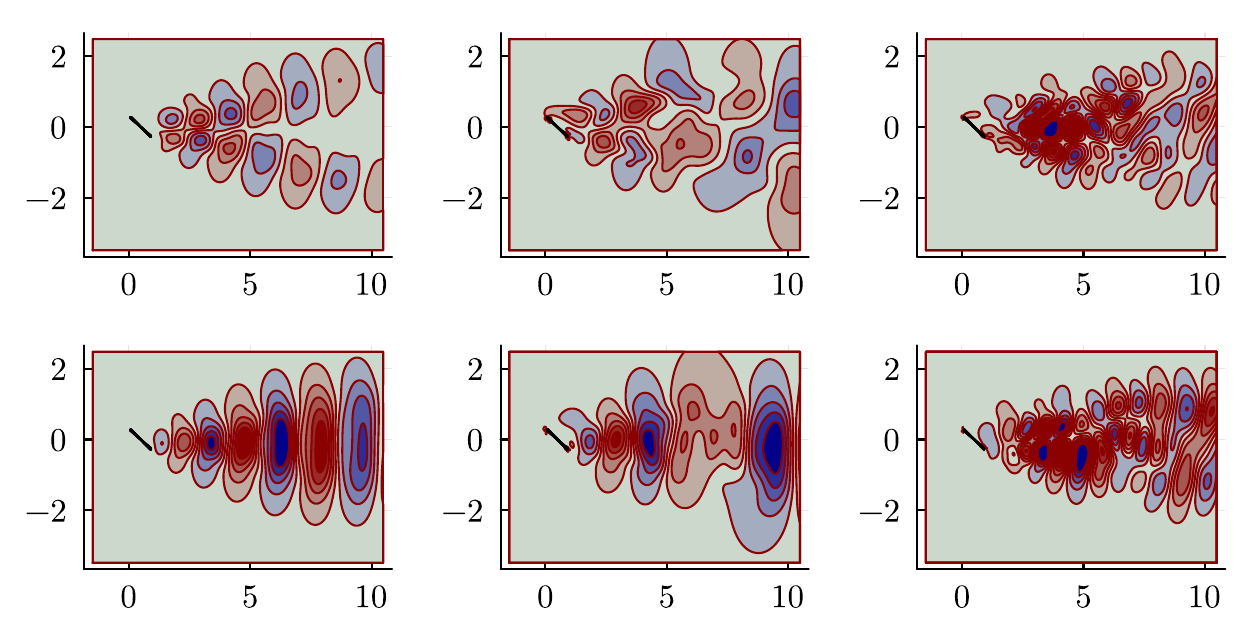}
	\end{subfigure}
	\hfill
	\begin{subfigure}{0.13\textwidth}
		\centering
		\includegraphics[width=\linewidth]{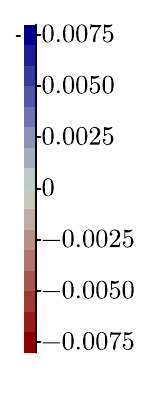}
	\end{subfigure}
	\caption{
  		Left singular modes for the low-Reynolds-number pitching-airfoil wake. Columns show modes $1$, $4$, and $16$. The top row gives the $x$-velocity component and the bottom row gives the $y$-velocity component. The airfoil is shown in black. The leading modes resolve large coherent wake structures, while the higher mode captures finer downstream oscillations.
	}\label{fig:airfoil}
\end{figure}

\begin{figure}
    \centering
    \begin{subfigure}{0.49\textwidth}
        \centering
        \includegraphics[width=\textwidth]{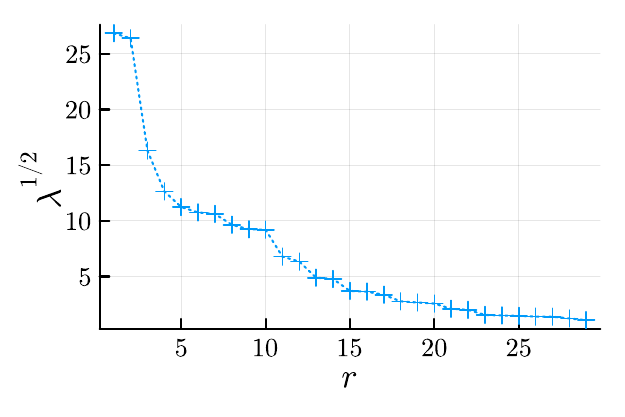}
    \end{subfigure}
    \hfill
    \begin{subfigure}{0.49\textwidth}
        \centering
        \includegraphics[width=\textwidth]{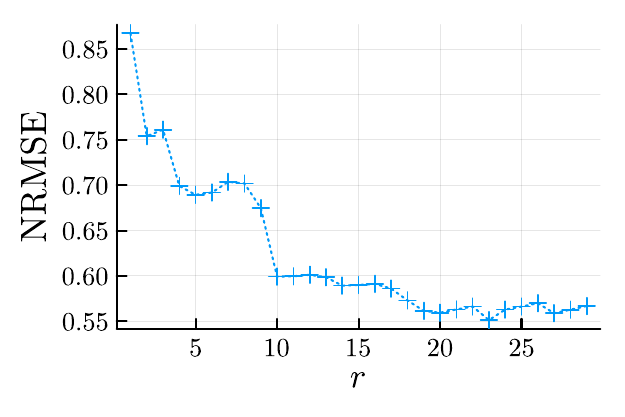}
    \end{subfigure}
    \caption{
        Rank dependence for the pitching-airfoil model. Left: singular values $\lambda^{1/2}$ of the Petrov--Galerkin regression, ordered by rank. Right: normalized root mean square error for five-step-ahead prediction on held-out snapshots. The prediction error drops rapidly as the leading modes are included and then saturates after roughly $r = 10$, consistent with the decay of the singular spectrum.
    }
    \label{fig:airfoil_error}
\end{figure}

\section{Conclusion}

We have framed a broad class of data-driven dynamics methods as Petrov--Galerkin regressions between observable spaces. The key change---allowing the two observable spaces to differ---puts DMD,
EDMD, SINDy, and Koopman regression into the same Petrov--Galerkin framework.
The regression
$
        \Psi(\Xv)K \approx \Phi(\Yv)
$
is a
finite section of the Koopman action between two chosen observable spaces; when
\(\Phi=id\), it is also a nonlinear predictor of the state obtained by
reevaluating the lift at each step.  In the large-data limit the fitted maps
converge to the corresponding projected operator, and the predictor converges
in \(L^2\) as the trial space is enriched.

The rectangular viewpoint also changes what should be plotted and interpreted.
Eigenvectors belong to square maps from a space to itself.  A Petrov--Galerkin
section between different observable spaces asks a different question: which
observable combinations are carried by the dynamics into the chosen trial
space, and with what strength?  Singular modes answer this question directly.
The right modes identify the inputs that are best resolved after evolution; the
left modes are their projected futures; the singular values measure the
coupling.  In this sense the singular value decomposition is not merely a
compression device.  It is the modal decomposition appropriate to finite,
asymmetric Koopman computation.

The examples show the range of this interpretation.  For the chaotic cubic map,
trajectory agreement ends at the Lyapunov horizon, but the learned map still
captures invariant statistics.  For the double gyre, singular modes recover
coherent transport structures.  For Lorenz--63, nonlinear lifting gives a compact
predictor that preserves the geometry and dominant statistics of the attractor.
For the pitching-airfoil wake, the singular spectrum supplies a practical rank
selection rule and the leading modes resolve coherent flow structures.  The
same least-squares object and singular modes therefore supports prediction, statistics, transport
diagnostics, and reduced modelling.

The message is simple.  Do not force the data-driven Koopman problem to be
square.  Choose the observables whose future matters, choose the space in which
that future can be represented, fit the Petrov--Galerkin map, and read its
singular structure.  This is \shortname{}: a rectangular, regression-based view
of nonlinear dynamics in which forecasting and interpretation are two sides of
the same matrix.

\section*{Acknowledgements}
MJC would like to thank Andrew Stuart for discussions about composing EDMD with nonlinear dictionary lifting. AH acknowledges that the project was supported by G-Research. 

\linespread{0.94}

\enlargethispage{20pt}


\vskip2pc

\bibliographystyle{plain}
\bibliography{sample}

\appendix

\section{Analytic Computations of Singular Modes for Cat Map}
\label{sec:arnold_sindy}

We consider the inverse Arnold cat map 
\begin{equation*}
	F^{-1}\begin{pmatrix} x \\ y \end{pmatrix} = \begin{cases}
		\begin{pmatrix} x - y \\ 2y - x + 1 \end{pmatrix}
			& \text{if } 0 \leq \dfrac{y}{x} < \dfrac{1}{2} \\[12pt]
		\begin{pmatrix} x - y \\ 2y - x \end{pmatrix}
			& \text{if } \dfrac{1}{2} \leq \dfrac{y}{x} < 1 \\[12pt]
		\begin{pmatrix} x + 1 - y \\ 2y - x - 1 \end{pmatrix}
			& \text{if } y \geq \dfrac{x}{2} + \dfrac{1}{2} \\[12pt]
		\begin{pmatrix} x - y + 1 \\ 2y - x \end{pmatrix}
			& \text{if } x \leq y < \dfrac{x}{2} + \dfrac{1}{2}
	\end{cases}
\end{equation*}
where $(x,y) \in [0,1]^2$ here. We transplant this map to $[-1/2,1/2]^2$ to apply \cref{alg:one_side}. Choosing $\Phi = id$, we desire an (ideally orthonormal) set of functions $\Psi$ such that $\Pi_\Psi F^{-1} = F^{-1}$. This is easily computed by performing Gram-Schmidt on $\{ [F^{-1}]_1, [F^{-1}]_2 \} \subset L^2 ([-1/2,1/2]_\text{per}^2)$. Now $\Pi_\Psi \cK \cP_\Phi$ can be represented by the matrix  
\begin{equation*}
	K_{i j} = A_{i j} = \left( \langle \psi_i, \cK \phi_j \rangle \right)_{\substack{1 \leq i \leq 2 \\ 1 \leq j \leq 2}} 
	= \begin{pmatrix}
		\tfrac{1}{2 \sqrt{3}} & \tfrac{1}{\sqrt{3}} \\ 
		0 & \sqrt{\tfrac{11}{12}}
	\end{pmatrix}  
\end{equation*}
with the singular value decomposition
{\scriptsize
\begin{equation*}
	K = \begin{pmatrix}
		\dfrac{-3+\sqrt{53}}{\sqrt{106-6\sqrt{53}}}
		&
		-\dfrac{3+\sqrt{53}}{\sqrt{106+6\sqrt{53}}}
		\\[18pt]
		\sqrt{\dfrac{1}{2}+\dfrac{3}{2\sqrt{53}}}
		&
		\sqrt{\dfrac{1}{2}-\dfrac{3}{2\sqrt{53}}}
	\end{pmatrix}
	\begin{pmatrix}
		\dfrac{1}{2}\sqrt{\dfrac{8+\sqrt{53}}{3}} & 0 \\[18pt]
		0 & \dfrac{1}{2}\sqrt{\dfrac{8-\sqrt{53}}{3}}
	\end{pmatrix}
	\begin{pmatrix}
		\dfrac{-7+\sqrt{53}}{\sqrt{2(53-7\sqrt{53})}}
		&
		-\dfrac{7+\sqrt{53}}{\sqrt{2(53+7\sqrt{53})}}
		\\[18pt]
		\sqrt{\dfrac{1}{2}+\dfrac{7}{2\sqrt{53}}}
		&
		\sqrt{\dfrac{2}{53+7\sqrt{53}}}
	\end{pmatrix}
\end{equation*}
}
The left singular modes $\varphi_1, \varphi_2$ are linear functions represented by the left singular vectors of $K$. Their direction is represented by arrows in \cref{fig:arnold_sindy}. The corresponding right singular modes are shown alongside.

\end{document}